\RequirePackage{amsmath}
\documentclass[a4paper, 10pt]{amsart}
\usepackage{amssymb, url, color, pb-diagram, graphicx, amscd, pb-diagram,mathrsfs}
\usepackage[colorlinks=true, bookmarks=true, pdfstartview=FitH, pagebackref=true]{hyperref}
\usepackage[nodayofweek]{datetime}
\usepackage{enumerate,stmaryrd} 
\usepackage{euscript}
\usepackage{verbatim}
\usepackage{graphicx}
\usepackage{times}
\usepackage{setspace}
\usepackage[space, compress, sort]{cite}
\usepackage{empheq}

\numberwithin{equation}{section}

\newtheorem{step}{Step}

\newtheorem{thm}{Theorem}[section]
\newtheorem{theorem}[thm]{Theorem}

\newtheorem{lemma}[thm]{Lemma}
\newtheorem{prop}[thm]{Proposition}

\theoremstyle{remark}

\newcommand{\ghat}{{\widehat{g}}}
\newcommand{\Khat}{\widehat{K}}

\newcommand{\phibar}{\overline{\varphi}}
\newcommand{\Cbar}{\overline{C}}
\newcommand{\Ibar}{\overline{I}}

\newcommand{\Atil}{\widetilde{A}}
\newcommand{\Wtil}{\widetilde{W}}

\newcommand{\phitil}{\widetilde{\phi}}
\newcommand{\pitil}{\widetilde{\pi}}

\newcommand{\sigmatil}{\widetilde{\sigma}}

\newcommand{\cB}{\mathcal{B}}
\newcommand{\cM}{\mathcal{M}}
\newcommand{\cR}{\mathcal{R}}

\newcommand{\psihat}{\widehat{\psi}}
\newcommand{\pihat}{\widehat{\pi}}

\newcommand{\bR}{\mathbb{R}}
\newcommand{\bL}{\mathbb{L}}

\newcommand{\ric}{\mathrm{Ric}}
\newcommand{\scal}{\mathrm{Scal}}

\DeclareMathOperator{\tr}{tr}
\DeclareMathOperator{\divg}{div}
\DeclareMathOperator{\hess}{Hess}
\DeclareMathOperator{\vol}{Vol}

\def\into{\hookrightarrow}

\newcommand{\definedas}{\mathrel{\raise.095ex\hbox{\rm :}\mkern-5.2mu=}}

\let\<\langle 
\let\>\rangle

\renewcommand{\leq}{\leqslant}
\renewcommand{\geq}{\geqslant}

\begin{document}

\title
{Solutions to the Einstein-scalar field constraint equations with a small
TT-tensor}

\begin{abstract}
In this paper, we prove a far-from-CMC result similar to
\cite{HNT1, HNT2, MaxwellNonCMC, GicquaudNgo} for the conformal
Einstein-scalar field constraint equations on compact Riemannian
manifolds with positive (modified) Yamabe invariant.
\end{abstract}

\author[R. Gicquaud]{Romain Gicquaud}
\address[R. Gicquaud]{
  Laboratoire de Math\'ematiques et de Physique Th\'eorique\\
  Universit\'e Fran\c cois Rabelais de Tours\\
  Parc de Grandmont\\ 37200 Tours \\ FRANCE}
\email{romain.gicquaud@lmpt.univ-tours.fr}

\author[C. Nguyen]{The Cang Nguyen}
\address[C. Nguyen]{
  Laboratoire de Math\'ematiques et de Physique Th\'eorique\\
  Universit\'e Fran\c cois Rabelais de Tours\\
  Parc de Grandmont\\ 37200 Tours \\ FRANCE}
\email{The-Cang.Nguyen@lmpt.univ-tours.fr}

\date{\today}


\maketitle

\tableofcontents

\section{Introduction}\label{secIntro}
Finding initial data for the Cauchy problem in general relativity is a
topical issue. The Einstein equation impose restrictions on the choice
of initial data called the constraint equations
and finding (physically relevant) solutions to them is the first step
in understanding the Cauchy problem in general relativity. We refer
the reader to \cite{BartnikIsenberg} for an extensive description of
the constraint equations.

Having in mind the geometric nature of general relativity, initial data
for the Cauchy problem are usually given as a triple $(M, \ghat, \Khat)$,
where $M$ is a manifold (which we will assume closed, i.e. compact
without boundary, for simplicity),
$\ghat$ is a Riemannian metric on $M$ and $\Khat$ is a symmetric
2-tensor on $M$. The Cauchy problem in general relativity then
consists in the following:

Find a space-time, i.e. a Lorentzian manifold $(\cM, h)$ solving
the Einstein equation
\[
 \ric^{h}_{\mu\nu} - \frac{1}{2} \scal^{h} h_{\mu\nu}
  = T_{\mu\nu},
\]
together with an embedding $M \hookrightarrow \cM$ such that $M$ becomes
a Cauchy surface with induced metric $\ghat$ and second fundamental form
$\Khat$.

Here $T$ is the so-called stress-energy tensor of the non-gravitational
fields (e.g. matter fields, electromagnetic fields...) one also wants to
encompass in the description of the universe.

During the past decades, lots of effort have been dedicated to the study
of these equations. However, until fairly recently, the methods could
construct only constant mean curvature or almost constant mean
curvature initial data. They can be subdivided into two main categories.
The first one is the conformal method and its variants which is the
method we will use in this article. The second one is the gluing
technique introduced by Corvino and Schoen \cite{CorvinoSchoen}.

The conformal method and the closely related conformal thin sandwich
method were historically the first methods introduced. The conformal
method will be described in Section \ref{secConformal}. We also refer
the reader to the very nice work of Maxwell giving a geometric
interpretation of this method
\cite{MaxwellConformalMethod, MaxwellInitialData, MaxwellKasner}.
The most stricking result of the conformal method is certainly the
classification of the set of vacuum ($T \equiv 0$) constant mean
curvature (CMC) solutions to the constraint equations achieved in 1995 by
Isenberg in \cite{Isenberg} and relies on the solution of the Yamabe
problem, see e.g. \cite{LeeParker}.

The near-CMC case was addressed soon after. We refer the reader to
\cite{BartnikIsenberg} for references.

The far-from-CMC case appears however much more difficult to tackle.
It was only in 2008 that Holst, Nagy and Tsogtgerel found a method
to construct solutions to the equations of the conformal method with
arbitrarily prescribed mean curvature. See \cite{HNT1, HNT2}. The
method was then extended by Maxwell \cite{MaxwellNonCMC} to the vacuum
case. Extension to the asymptotically Euclidean case was proven in
\cite{DiltsIsenbergMazzeoMeier}.
Another point of view on this method is given by the first author and Ngo
in \cite{GicquaudNgo}. While the result of \cite{GicquaudNgo} is weaker
from a mathematical point of view than the one in \cite{MaxwellNonCMC},
the proof appears to be constructive.

The difficulty of the equations of the conformal method in the
far-from-CMC case lies in the difficulty to obtain an a priori estimate
on the solutions to the equations there is to solve. The method
designed by Dahl, Humbert and the first author in \cite{DahlGicquaudHumbert}
can be understood as a criterion for the existence of such an a priori
estimate. This method turns out to be particularly efficient with
negatively (Ricci) curved metrics, see \cite{GicquaudSakovich}.
See also \cite{DiltsLeach} for the asymptotically cylindrical case.

Simplified proofs of the results in \cite{Maxwell,DahlGicquaudHumbert}
are given by the second author in \cite{Nguyen}. A comparative point on
view of both methods is given in \cite{GicquaudNgo}.

Introducing non-gravitational fields in the constraint equations
usually leads to not so much more difficult equations to solve for the
conformal method because the terms that appear are non critical. There is
however an important counterexample to this which is scalar fields. The
potential of the field (which also encodes the cosmological constant)
can change the sign of one of the dominant terms in the Lichnerowicz
equation, see \eqref{eqLichnerowicz}.
In particular, the method of \cite{DahlGicquaudHumbert} cannot work any
longer and the Lichnerowicz equation may admits multiple solutions,
see e.g. \cite{ChruscielGicquaud}, \cite{PremoselliMultiplicity} and
references therein.

The CMC case for the Einstein-scalar field constraint equations was
studied in \cite{CBIP072} and \cite{HebeyPacardPollack}. The near CMC
case was addressed by Premoselli in \cite{Premoselli}. We would also like
to refer the reader to \cite{CBIP06} and \cite{Sakovich} for similar
results in the non-compact cases.\\

Based on the ideas developed in \cite{GicquaudNgo},\cite{Nguyen} and
\cite{Premoselli}, we show that the method of Holst et al. can be
extended to the Einstein-scalar field constraint equations.\\
The outline of the article is as follows. In Section \ref{secPrelim}
we given the constraint equations with a scalar field and we introduce
the conformal method in this context. We then present the extension of
the result of \cite{GicquaudNgo} in Section \ref{secGicquaudNgo}, see
Theorem \ref{thmGicquaudNgo}. Finally, we address the much more difficult
extension of the method of \cite{HNT2} in Section \ref{secSmallTT}, see
Theorem \ref{thmHolst}.
In the course of the proof, we prove Theorem \ref{thmLich} which shows
existence of solutions to the Lichnerowicz equation in our context. 

The outline of the article is as follows. In Section \ref{secPrelim}
we give the constraint equations with a scalar field and we introduce
the conformal method in this context. We then present the extension of
the result of \cite{GicquaudNgo} in Section \ref{secGicquaudNgo}, see
Theorem \ref{thmGicquaudNgo}. Finally, we address the much more difficult
extension of the method of \cite{HNT2} in Section \ref{secSmallTT}, see
Theorem \ref{thmHolst}.
In the course of the proof, we prove Theorem \ref{thmLich} which shows
existence of solutions to the Lichnerowicz equation in our context.\\

\noindent\textbf{Acknowledgements:} The authors are grateful to Emmanuel
Humbert for his support and his careful proofreading of a preliminary
version of this article. The authors also warmly thank Laurent V\'eron
for useful discussions.

\section{Preliminaries}\label{secPrelim}

\subsection{The constraint equations}\label{secConstraints}

We first recall the derivation of the constraint equations for the
Einstein equations with a scalar field $\Psi$, refering to
\cite{BartnikIsenberg,ChoquetBruhat,CBIP072} for further information.
The stress-energy tensor of $\Psi$ reads
\[
 T_{\mu\nu} = \nabla_\mu \Psi \nabla_\nu \Psi
  - \left(\frac{1}{2} \left|d \Psi\right|^2_h + V(\Psi)\right) h_{\mu\nu},
\]
where $h$ denotes the space-time metric and $V$ is the potential of
the scalar field. The Einstein-scalar field equations for $h$ and $\phi$
are then
\begin{equation}\label{eqEinstein}
 \left\lbrace
 \begin{aligned}
   \ric^{h}_{\mu\nu} - \frac{1}{2} \scal^{h} h_{\mu\nu}
     & = T_{\mu\nu},\\
   \square_h \Psi & = V'(\Psi).
 \end{aligned}
 \right.
\end{equation}

Let $M$ be a Cauchy surface in the space-time $(\cM, h)$ and let $\nu$
denotes its unit future-pointing timelike normal. If $\ghat$ is the
metric induced by $h$ on $M$ and if $\Khat(X, Y) = h(X, \nabla_{Y} \nu)$
is the second fundamental form of $M$, then contracting the first
equation in \eqref{eqEinstein} with $\nu$ twice gives the Hamiltonian
constraint:
\begin{equation}\label{eqHamiltonian}
 \scal^{\ghat} + \left(\tr_{\ghat} \Khat\right)^2 - \left|\Khat\right|^2_{\ghat}
  = \pihat^2 + \left|d \psihat\right|^2_{\ghat} + 2 V(\psi),
\end{equation}
where $\psihat = \Psi\vert_M$ is restriction of $\Psi$ to $M$ and 
$\pihat = \nabla_\nu \Psi$ is the time derivative of $\Psi$.

Contracting only once the first equation in \eqref{eqEinstein} and
restricting the remaining free index to spatial directions (i.e. tangent
to $M$), we get the momentum constraint:
\begin{equation}\label{eqMomentum}
 \divg^{\ghat} \Khat - d (\tr_{\ghat} \Khat) = \pihat d\psihat.
\end{equation}

Hence, solving the constraint equations \eqref{eqHamiltonian}-
\eqref{eqMomentum} appears to be a necessary condition if one hopes to
solve the Einstein equations. Conversly, the celebrated work of
Choquet-Bruhat \cite{CB1} and subsequently of Choquet-Bruhat and Geroch
\cite{CB2} ensure that any 5-tuple $(M, \ghat, \Khat, \psi, \pihat)$
satisfying \eqref{eqHamiltonian}-\eqref{eqMomentum} yields a unique
solution of the Einstein-scalar field equations \eqref{eqEinstein},
see also \cite{CBIP07}. We
refer the reader to \cite{ChoquetBruhat} and
\cite{RingstromCauchyProblem} for a comprehensive introduction to the
Cauchy problem in general relativity.

\subsection{The conformal method}\label{secConformal}

We assume from now on that the manifold $M$ is a given closed manifold of
dimension $n$. Counting the degrees of freedom of the 4-tuple
$(\ghat, \Khat, \psi, \pihat)$ and comparing it to the number of equations
provided by \eqref{eqHamiltonian} and \eqref{eqMomentum}, we immediately
see that the constraint equations form a (very) underdetermined system. As
usual in treating such type of problems, we decompose the variables
$(\ghat, \Khat, \psihat, \pihat)$ into given data and unknowns that have
to be adjusted to fulfill the constraint equations.

The splitting we will use together with the regularity we will assume are
the following (here $p > n$ is given):

\begin{itemize}
 \item \textbf{Given (seed) data}:
  \begin{itemize}
   \item A (background) metric $g \in W^{2, \frac{p}{2}}$,
   \item A function $\tau: M \to \bR$, $\tau \in W^{1, p}$,
   \item Two functions $\psi \in W^{1, p}$ and $\pi \in L^p$,
   \item A symmetric traceless and divergence free 2-tensor
    $\sigma \in W^{1, p}$,
  \end{itemize}
 \item \textbf{Unknowns}:
  \begin{itemize}
   \item A positive function $\phi \in W^{2, \frac{p}{2}}$,
   \item A 1-form $W \in W^{2, \frac{p}{2}}$.
  \end{itemize}
\end{itemize}

From these data, we cook up the initial data as follows:

\[
 \ghat   = \phi^{N-2} g,\quad
 \Khat   = \frac{\tau}{n} \phi^{N-2} g + \phi^{-2} (\sigma + \bL W),\quad
 \psihat = \psi,\quad
 \pihat  = \phi^{-N} \pi.
\]

We have used the following notations: $N \definedas \frac{2n}{n-2}$ and $\bL$
is the conformal Killing operator acting on 1-forms, namely, in coordinates
\[
 \bL W_{ij} = \nabla_i W_j + \nabla_j W_i - \frac{2}{n} \nabla^k W_k g_{ij},
\]
where $\nabla$ is the Levi-Civita connection associated to the metric $g$.

The constraint equations \eqref{eqHamiltonian}-\eqref{eqMomentum} can be
rewritten in terms of these new variables:

\begin{subequations}\label{eqConformalConstraints}
\begin{align}
 \label{eqLichnerowicz}
 \frac{4(n-1)}{n-2} \Delta \phi + \cR_\psi \phi
  & = \cB_{\tau, \psi} \phi^{N-1}
    + \frac{\left|\sigma + \bL W\right|^2 + \pi^2}{\phi^{N+1}},\\
 \label{eqVector}
 - \frac{1}{2} \bL^* \bL W & = \frac{n-1}{n} \phi^N d\tau - \pi d\psi.
\end{align}
\end{subequations}

Equation \eqref{eqLichnerowicz} is usually named the Lichnerowicz
equation, while Equation \eqref{eqMomentum} is usually refered to as
the vector equation. Our convention for the Laplacian is
\[
 \Delta \phi \definedas - g^{ij} \nabla_i \nabla_j \phi,
\]
and the operator appearing in \eqref{eqVector} is the vector Laplacian:
\[
 - \frac{1}{2} \bL^* \bL W_j = \nabla^i \bL W_{ij}
\]

The functions $\cR_\psi$ and $\cB_{\tau, \psi}$ that appear in the
Lichnerowicz equation are given by:

\[
 \cR_\psi = \scal^g - \left|d\psi\right|^2_g, \quad
 \cB_{\tau, \psi} = - \frac{n-1}{n} \tau^2 + 2 V(\psi).
\]

Compared to the vacuum case (i.e. $\psi, \pi \equiv 0$), the coefficient
$\cB_{\tau, \psi}$ can have arbitrary sign. Also, even if the
metric $g$ has positive Yamabe invariant, the generalized conformal
Laplacian
\begin{equation}\label{eqModifiedConformalLaplacian}
 L_{g,\psi}: \phi \mapsto \frac{4(n-1)}{n-2} \Delta \phi + \cR_\psi \phi
\end{equation}
is not necessarily a coercive operator, meaning that there may not
exists a constant $c > 0$ such that
\[
 \forall \phi \in W^{1, 2}, \int_M \phi L_{g,\psi} \phi d\mu^g \geq c \left\|\phi\right\|_{W^{1, 2}}^2.
\]
This assumption however will turn out to be very important in our analysis
and plays a role analog to the assumption that the metric $g$ has positive
Yamabe invariant in \cite{HNT1, HNT2, MaxwellNonCMC}. Another important
assumption we will need is that $(M, g)$ has no non-zero conformal Killing
vector field. This assumption is generically true, see \cite{BCS05}.

\section{An implicit function argument}\label{secGicquaudNgo}

In this section, we show that the method introduced in \cite{GicquaudNgo}
can be straightforwardly generalized to the system
\eqref{eqConformalConstraints}.

\begin{theorem}\label{thmGicquaudNgo}
 Let $(M, g)$ a closed Riemannian manifold, $\tau \in W^{1, p}$,
 $\psi \in W^{1, p}$, $\pitil \in L^p$ and $\sigmatil \in L^p$ be
 given. Assume further that the operator $L_{g, \psi}$ defined in
 \eqref{eqModifiedConformalLaplacian} is coercive and that $(M, g)$ has
 no non-zero conformal Killing vector field. There exists an
 $\epsilon_0 > 0$ such that for all $\epsilon \in (0, \epsilon_0)$,
 the system \eqref{eqConformalConstraints} with
 \[
 \sigma \equiv \epsilon \sigmatil, \quad \pi \equiv \epsilon \pi
 \]
 has a solution $(\phi, W) \in W^{2, \frac{p}{2}}\times W^{2, \frac{p}{2}}$ with $\phi > 0$.
\end{theorem}

As in the article \cite{GicquaudNgo}, we divide the proof into several
steps:

\setcounter{step}{-1}
\begin{step}\label{stVector0}
 There exists a unique solution $\Wtil_0 \in W^{2, \frac{p}{2}}$ to
 \begin{equation}\label{eqVector0}
  - \frac{1}{2} \bL^* \bL W = - \pitil d\psi.
 \end{equation}
\end{step}

\begin{proof} The argument is standard, see e.g.
\cite[Proposition 5]{MaxwellNonCMC}.
Note that $\pitil d\psi \in L^{\frac{p}{2}}$. The operator
\[
 -\frac{1}{2} \bL^* \bL: W^{2, \frac{p}{2}} \to L^{\frac{p}{2}}
\]
is Fredholm with zero index. Its kernel is, by a simple integration by
parts argument, the set of conformal Killing vector fields which is
reduced to $\{0\}$ by assumption. Hence $-\frac{1}{2} \bL^* \bL$ is an
isomorphism.
\end{proof}

\begin{step}\label{stLichnerowicz0}
 There exists a unique solution $\phitil_0 \in W^{2, \frac{p}{2}}$ to the
 following equation:
 \begin{equation}\label{eqLichnerowicz0}
  \frac{4(n-1)}{n-2} \Delta \phi + \cR_\psi \phi = \frac{\left|\sigmatil + \bL \Wtil_0\right|^2 + \pitil^2}{\phi^{N+1}},\\
 \end{equation}
\end{step}

\begin{proof}
 We set
 \[
  \Atil \definedas \left|\sigmatil + \bL \Wtil_0\right|^2 + \pitil^2
 \]
 for convenience. Since $\Wtil_0 \in W^{2, \frac{p}{2}}$,
 $\bL \Wtil_0 \in W^{1, \frac{p}{2}} \into L^p$. Indeed, from the Sobolev
 injection, $W^{1, \frac{p}{2}} \into L^p$, where
 \[
  q = \frac{np}{2n-p} > p
 \]
 (here we assumed that $p < 2n$). It follows that $\Atil \in L^{\frac{p}{2}}$.
 We first prove that there exists a unique positive solution $\phibar$ to
 \begin{equation}\label{eqLichnerowiczSuper}
  \frac{4(n-1)}{n-2} \Delta \phi + \cR_\psi \phi = \Atil.
 \end{equation}
 We remark that, integrating the righthand side, we get
 \begin{align*}
  \int_M \Atil d\mu^g
    & = \int_M \left(\left|\sigmatil\right|^2 + \pitil^2\right) d\mu^g  + \int_M \left|\bL W\right|^2 d\mu^g\\
    & \geq \int_M \left(\left|\sigmatil\right|^2 + \pitil^2\right) d\mu^g\\
    & > 0.
 \end{align*}
 We rely on the Lax-Milgram theorem. Since $L_{g, \psi}$ is coercive, there
 exists a unique weak solution $\phibar$ to \eqref{eqLichnerowiczSuper}
 which is uniquely characterized by
 \[
  \int_M \left(\frac{2(n-1)}{n-2} \left|d\phibar\right|^2 + \frac{\cR_\psi}{2} \phibar^2 - \Atil \phibar\right) d\mu^g
    = \min_{\phi \in W^{1, 2}} F(\phi),
 \]
 where
 \[
  F(\phi) \definedas \int_M \left(\frac{2(n-1)}{n-2} \left|d\phi\right|^2 + \frac{\cR_\psi}{2} \phi^2 - \Atil \phi\right) d\mu^g.
 \]
 Since $\Atil \geq 0$, we have $F(|\phi|) \leq F(\phi)$ for any
 $\phi \in W^{1, 2}$. As a consequence $\phibar$ being the unique
 minimizer of $F$, $\phibar \geq 0$. By elliptic regularity, we have
 that $\phibar \in W^{2, \frac{p}{2}}$. In particular, $\phibar$ is continuous.
 It can be argued by contradiction that $\phibar > 0$. Indeed, if the
 set $\Omega = \{\phibar = 0\}$ was not empty, it would follows from
 the Harnack inequality we borrow from
 \cite[Theorem 9.22]{GilbargTrudinger} applied to $u = \phibar$ and
 $f \equiv 0$ in a ball $B_R$ centered at a boundary point of $\Omega$
 that $\phibar \equiv 0$ on $B_R$ which is a contradiction.

 Setting $a \definedas \min_M \phibar$, $b \definedas \max_M \phibar$,
 one can readily check that the function
 \[
  \phibar_+ \definedas a^{-\frac{N+1}{N}} \phibar, \text{ resp . } \phibar_- \definedas b^{-\frac{N+1}{N}} \phibar,
 \]
 is a supersolution (resp. a subsolution) for Equation
 \eqref{eqLichnerowicz0}. Existence of a solution to
 \eqref{eqLichnerowicz0} follows then from the standard sub- and
 supersolution method, see e.g. \cite[Lemma 3.4]{GicquaudHuneau} or
 \cite{MaxwellRoughCompact}. Uniqueness of $\phitil_0$ is also classical,
 see \cite{DahlGicquaudHumbert}. However, here we can simply remark that
 the functional
 \[
  G(\phi) \definedas \int_M \left(\frac{2(n-1)}{n-2} \left|d\phi\right|^2 + \frac{\cR_\psi}{2} \phi^2
   + \frac{1}{N} \frac{\Atil}{\phi^N}\right) d\mu^g
 \]
 is strictly convex on the set of positive $H^1$-functions (i.e.
 so that there exists $\epsilon > 0$ such that $\phi \geq \epsilon$ a.e.)
 Its critical points being exactly the solutions to
 \eqref{eqLichnerowicz0}, we conclude that the solution to
 \eqref{eqLichnerowicz0} is unique. This idea will be developed further
 in Section \ref{secLichnerowicz}.
\end{proof}

\begin{step}\label{stDeformation}
 There exists $\epsilon > 0$ and a $C^1$-map
 \[
  \begin{array}{ccc}
   [0, \epsilon) & \to & W^{2, \frac{p}{2}} \times W^{2, \frac{p}{2}}\\
   \lambda & \mapsto & (\phitil_\lambda, \Wtil_\lambda)
  \end{array}
 \]
 such that
 \begin{itemize}
  \item $\phitil_\lambda$ and $\Wtil_\lambda$ solve
   \begin{subequations}\label{eqConformalConstraints1}
    \begin{align}
     \label{eqLichnerowicz1}
     \frac{4(n-1)}{n-2} \Delta \phitil_\lambda + \cR_\psi \phitil_\lambda
      & = \lambda^2 \cB_{\tau, \psi} \phitil_\lambda^{N-1} + \frac{\left|\sigmatil + \bL \Wtil_\lambda\right|^2 + \pitil^2}{\phitil_\lambda^{N+1}},\\
     \label{eqVector1}
     - \frac{1}{2} \bL^* \bL \Wtil_\lambda & = \frac{n-1}{n} \lambda \phitil_\lambda^N d\tau - \pitil d\psi.
    \end{align}
   \end{subequations}
  \item $\phitil_\lambda \to \phitil_0$ and $\Wtil_\lambda \to \Wtil_0$
   when $\lambda \to 0$, where $\Wtil_0$ and $\phitil_0$ are as
   defined in Steps \ref{stVector0} and \ref{stLichnerowicz0}.
 \end{itemize}
\end{step}

Note that Equations \eqref{eqConformalConstraints1} interpolate between
the original conformal constraint equations \eqref{eqConformalConstraints}
when $\lambda = 1$ and Equations \eqref{eqLichnerowicz0}-\eqref{eqVector0}
when $\lambda = 0$.

\begin{proof}
The proof is via the implicit function theorem. Let
$\Phi: \bR \times W^{2, \frac{p}{2}} \times W^{2, \frac{p}{2}} \to L^\frac{p}{2} \times L^\frac{p}{2}$ be the
following operator:

\[
 \Phi_\lambda \begin{pmatrix} \phitil\\ \Wtil\end{pmatrix}
  \mapsto
 \begin{pmatrix}
  \frac{4(n-1)}{n-2} \Delta \phitil + \cR_\psi \phitil
     - \lambda^2 \cB_{\tau, \psi} \phitil^{N-1} - \frac{\Atil}{\phitil^{N+1}}\\
  - \frac{1}{2} \bL^* \bL \Wtil - \frac{n-1}{n} \lambda \phitil^N d\tau + \pitil d\psi
 \end{pmatrix}.
\]

Its differential with respect to the variables $(\phitil, \Wtil)$ at
$(\lambda = 0, \phitil_0, \Wtil_0)$ is given by the following block
upper triangular matrix:

\[
 D\Phi_{\lambda=0}(\phitil_0, \Wtil_0)
  =
  \begin{pmatrix}
   \frac{4(n-1)}{n-2} \Delta + \cR_\psi + (N+1) \frac{\left|\sigmatil + \bL \Wtil_0\right|^2 + \pitil^2}{\phitil_0^{N+2}}
     & - \frac{2}{\phitil_0^{N+1}} \left\<\sigmatil + \bL \Wtil_0, \bL \cdot\right\>\\
   0 & - \frac{1}{2} \bL^* \bL \cdot
  \end{pmatrix}.
\]

Each diagonal block is Fredholm with zero index and has, by assumption,
a trivial kernel. This proves that $D\Phi_{\lambda=0}(\phitil_0, \Wtil_0)$
is invertible. The existence of the curve of solutions to
\eqref{eqConformalConstraints1} on some interval $[0, \epsilon)$
is then guaranteed by the implicit function theorem.
\end{proof}

The last step is a straightforward calculation.

\begin{step}\label{stRescaling}
 Let $(\phitil_\lambda, \Wtil_\lambda)$ be as in Step \ref{stDeformation}.
 Setting
 \[
  \left\lbrace
   \begin{aligned}
     \phi_\lambda & \definedas \lambda^{\frac{2}{N-2}} \phitil_\lambda,\\
     W_\lambda & \definedas \lambda^{\frac{N+2}{N-2}} \Wtil_\lambda,\\
     \sigma_\lambda & \definedas \lambda^{\frac{N+2}{N-2}} \sigmatil,\\
     \pi_\lambda & \definedas \lambda^{\frac{N+2}{N-2}} \pitil,
   \end{aligned}
  \right.
 \]
 then
 $(\phi_\lambda, W_\lambda)$ solves the system
 \eqref{eqConformalConstraints} with $\sigma = \sigma_\lambda$ and
 $\pi = \pi_\lambda$.
\end{step}

The proof of Theorem \ref{thmGicquaudNgo} follows by setting
$\epsilon_0 = \lambda_0^{\frac{N+2}{N-2}}$ and
$\epsilon = \lambda^{\frac{N+2}{N-2}}$.

\section{An existence result for \texorpdfstring{$\sigma$}{TEXT} and
\texorpdfstring{$\pi$}{TEXT} small in \texorpdfstring{$L^2$}{TEXT}}
\label{secSmallTT}

In this section, we adapt the method of \cite{HNT1, HNT2, MaxwellNonCMC}
to our context. The first step is to prove an existence result for
solutions to the Lichnerowicz equation. Very nice existence results for
solutions to the Lichnerowicz equation are given in \cite{HebeyPacardPollack},
\cite{Premoselli, PremoselliMultiplicity} and \cite{HebeyVeronelli}. We
prove here an existence result suited to our applications. See Theorem
\ref{thmLich}. We then study the full system \eqref{eqConformalConstraints}
and obtain the following theorem:

\begin{theorem}\label{thmHolst}
 Under the regularity assumptions stated in Section \ref{secConformal},
 and assuming that the operator $L_{g, \psi}$ is coercive and that
 $(M, g)$ has no conformal Killing vector fields, 
 the system \eqref{eqConformalConstraints} admits at least one solution
 $(\phi, W) \in W^{2, \frac{p}{2}} \times W^{2, \frac{p}{2}}$ provided that
 \[
  \int \left(|\sigma|^2 + \pi^2\right) d\mu^g
 \]
 is less than some small constant (depending on the seed data).
\end{theorem}

\subsection{The Lichnerowicz equation}\label{secLichnerowicz}

Here and in what follows, we define the following norm. Given
$\phi \in H^1(M, g)$, we set
\[
\left\|\phi\right\|_h^2 \definedas
  \int_M \left(\frac{4(n-1)}{n-2}\left|d\phi\right|^2 + \cR_\psi \phi^2\right) d\mu^g.
\]
Since we assumed that the modified conformal Laplacian is coercive, there
exists a constant $s > 0$ such that for any $\phi \in H^1$, we have
\begin{equation}\label{eqDefNormK}
\left\|\phi\right\|^2_h \geq s \left\|\phi\right\|_{L^N}^2.
\end{equation}
The aim of this section is to prove the following theorem:

\begin{theorem}\label{thmLich}
 Assuming that $\left|\sigma + \bL W\right|^2 + \pi^2 \in L^{\frac{p}{2}}$,
 there exists a (small) constant
 $\mu = \mu(s, \left\|\cB_{\tau, \psi}\right\|_{L^\infty}) > 0$ such that if
 \[
  0 < \int_M \left(\left|\sigma + \bL W\right|^2 + \pi^2\right) d\mu^g < \mu
 \]
 the Lichnerowicz equation \eqref{eqLichnerowicz} admits a solution
 $\phi \in H^1$ which is a stable minimizer for the functional
\begin{equation}\label{eqLichnerowiczFunctional}
\begin{aligned}
I_W(\phi) & \definedas
 \frac{1}{2} \int_M \left(\frac{4(n-1)}{n-2} \left|d\phi\right|^2 + \cR_\psi \phi^2\right) d\mu^g
 - \int_M \frac{\cB_{\tau, \psi}}{N} \phi^N d\mu^g\\
 & \qquad \qquad + \int_M \frac{\left|\sigma + \bL W\right|^2 + \pi^2}{N \phi^N} d\mu^g.
\end{aligned}
\end{equation}
 and whose energy satisfies
 \[
  \left\|\phi\right\|_h^2 \leq C \left(\int_M \left(\left|\sigma + \bL W\right|^2 + \pi^2\right) d\mu^g\right)^{\frac{2}{N+2}}.
 \]
 for some constant $C = C(s, \left\|\cB_{\tau, \psi}\right\|_{L^\infty}, \mu)$.
\end{theorem}

The spirit of the proof of this theorem is different from
\cite{HebeyPacardPollack}. The point being that we want to obtain
a stable solution $\phi_0$, meaning that $\phi_0$ is a stable local
minimum for the functional $I$ defined in \eqref{eqLichnerowiczFunctional},
while \cite{HebeyPacardPollack} uses the mountain pass lemma.
Stability will ensure that the minimum $\phi_0$ varies continously
with respect to the parameters. This will turn out to be very important
when applying the Schauder fixed point theorem in Section \ref{secCoupled}.

The proof of Theorem \ref{thmLich} will be carried out in the
remaining of this section. For convenience, we denote
\[
 A_W := \int_M \left(\left|\sigma + \bL W\right|^2 + \pi^2\right) d\mu^g.
\]
We also denote by $B_{R_0}$ the ball of radius $R_0 > 0$ centered at the
origin in $H^1$ for the norm $\|\cdot\|_h$.

\begin{lemma}\label{lmIbar}
 There exists an $R_0 > 0$ depending only on $g$ and $\psi$ such that
 the functional
 \begin{equation}\label{eqIbar}
  \Ibar(\phi) \definedas
  \frac{1}{2} \int_M \left(\frac{4(n-1)}{n-2} \left|d\phi\right|^2 + \cR_\psi \phi^2\right) d\mu^g
  - \int_M \frac{\cB_{\tau, \psi}}{N} |\phi|^N d\mu^g
 \end{equation}
 has $\hess \Ibar(\phi) (u, u) \geq \frac{1}{2} \left\|u\right\|_h^2$
 for all $\phi \in B_{R_0}(0)$ and all $u \in H^1$.
\end{lemma}

In particular, we have
\begin{equation}\label{eqEstimateIbar}
 \frac{1}{4} \left\|\phi\right\|_h^2 \leq \Ibar(\phi)
\end{equation}
for all $\phi \in B_{R_0}$.

\begin{proof}[Proof of Lemma \ref{lmIbar}]
The Hessian of $\Ibar$ at $\phi \in H^1$ and in the direction $u \in H^1$
is given by
\[
\hess \Ibar(\phi) (u, u) =
 \int_M \left[\frac{4(n-1)}{n-2} \left|du\right|^2 + \cR_\psi u^2
  - (N-1) \cB_{\tau, \psi} |\phi|^{N-2} u^2\right] d\mu^g.
\]
We estimate the Hessian as follows:
\begin{align*}
\hess \Ibar(\phi) (u, u)
 & \geq \int_M \left(\frac{4(n-1)}{n-2} \left|du\right|^2 + \cR_\psi u^2\right) d\mu^g
   - (N-1) \left\|\cB_{\tau, \psi}\right\|_{L^\infty} 
      \left\|\phi\right\|^{N-2}_{L^N} \left\|u\right\|^2_{L^N}\\
 & \geq \left\|u\right\|_h^2 - \frac{N-1}{s} \left\|\cB_{\tau, \psi}\right\|_{L^\infty} 
      \left\|\phi\right\|^{N-2}_{L^N} \left\|u\right\|^2_h\\
 & \geq \left(1 - \frac{N-1}{s} \left\|\cB_{\tau, \psi}\right\|_{L^\infty} 
      \left\|\phi\right\|^{N-2}_{L^N}\right) \left\|u\right\|^2_h\\
\end{align*}
Thus, if $\left\|\phi\right\|_{L^N} <
\left(\frac{s}{2(N-1) \left\|\cB_{\tau, \psi}\right\|_{L^\infty}}\right)^{\frac{1}{N-2}}$,
the Hessian of $\Ibar$ satisfies the assumptions of the lemma. From the
Sobolev embedding theorem, the conclusion of the lemma holds with
\begin{equation}\label{eqDefR0}
R_0 = s^{1/2} \left(\frac{s}{2(N-1) \cB_{\tau, \psi}}\right)^{\frac{1}{N-2}}.
\end{equation}
\end{proof}

We now introduce the following functional:

\begin{equation}\label{eqDefIepsilon}
 \begin{aligned}
 I^\epsilon_W(\phi) & \definedas \frac{1}{2} \int_M \left( \left|d\phi\right|^2 + \cR_\psi \phi^2\right) d\mu^g
 - \int_M \frac{\cB_{\tau, \psi}}{N} |\phi|^N d\mu^g\\
 & \qquad + \int_M \frac{A_W}{N (\phi+\epsilon)^N} d\mu^g
   + \int_M \phi_-^N d\mu^g,
 \end{aligned}
\end{equation}
where $\phi_- \definedas -\min \{\phi, 0\}$.
Note that the two terms we added are convex on the set
\begin{equation}\label{eqDefOmega}
\Omega_\epsilon \definedas \{\phi \in H^1, \phi \geq -\epsilon/2 \text{ a.e.}\}.
\end{equation}
This set is convex and closed for the $H^1$-norm. Indeed, we have
\[
 \Omega_\epsilon = \bigcap_f \left\{\phi \in H^1, \int_M f \phi d\mu^g \geq - \frac{\epsilon}{2} \int_M f d\mu^g\right\},
\]
where we took the intersection over the set of (say) continuous positive
functions $f$. In particular, the set $\Omega_\epsilon $
is compact for the weak topology on $H^1$.

Continuity of $I^\epsilon_W$ is easy to prove. Indeed, the only difficult
term to prove continuity of is
\[
 \int_M \frac{A_W}{N (\phi+\epsilon)^N} d\mu^g
\]

But, given $\phi_0 \in \Omega_{\epsilon}$ and $\nu > 0$, there exists
$\Lambda > 0$ so that
\[
  \frac{1}{N}\left(\frac{2}{\epsilon}\right)^{N}\int_{A_W \geq \Lambda} A_W d\mu^g \leq \frac{\nu}{4}.
\]
So, for any $\phi \in \Omega_\epsilon$, we have
\begin{align*}
 & \left|\int_M \frac{A_W}{N (\phi+\epsilon)^N} d\mu^g - \int_M \frac{A_W}{N (\phi_0+\epsilon)^N} d\mu^g\right|\\
 & \qquad \qquad \leq \left|\int_{A_W < \Lambda} \frac{A_W}{N (\phi+\epsilon)^N} d\mu^g - \int_{A_W < \Lambda} \frac{A_W}{N (\phi_0+\epsilon)^N} d\mu^g\right|\\
 & \qquad \qquad \qquad + \left|\int_{A_W \geq \Lambda} \frac{A_W}{N (\phi+\epsilon)^N} d\mu^g\right| + \left|\int_{A_W \geq \Lambda} \frac{A_W}{N (\phi_0+\epsilon)^N} d\mu^g\right|\\
 & \qquad \qquad \leq \left|\int_{A_W < \Lambda} \left(\frac{A_W}{N (\phi+\epsilon)^N} - \frac{A_W}{N (\phi_0+\epsilon)^N}\right) d\mu^g\right| + \frac{\nu}{2}\\
 & \qquad \qquad \leq \left|\int_{A_W < \Lambda} \left(\int_0^1\frac{1}{(t \phi + (1-t) \phi_0)^{N+1}}dt\right) (\phi-\phi_0) A_W d\mu^g\right| + \frac{\nu}{2}\\
 & \qquad \qquad \leq \Lambda \left(\frac{2}{\epsilon}\right)^{-N-1} \left\|\phi-\phi_0\right\|_{L^1} + \frac{\nu}{2}.
\end{align*}
Hence, provided $\left\|\phi-\phi_0\right\|_{L^1} < \frac{\nu}{2\Lambda}\left(\frac{2}{\epsilon}\right)^{N+1}$,
we have
\[
 \left|\int_M \frac{A_W}{N (\phi+\epsilon)^N} d\mu^g - \int_M \frac{A_W}{N (\phi_0+\epsilon)^N} d\mu^g\right|
  < \nu.
\]
The $H^1$-norm being stronger than the $L^1$-norm this concludes the
proof of the continuity of $I_W^\epsilon$.
Note that $I$ itself is continuous a priori only on
$\bigcup_{\epsilon < 0} \Omega_\epsilon$ which is not closed. This
is one of the reasons why we need to regularize $I$. 

Now note that since $I^\epsilon_W$ is (strictly) convex and continuous on
$\Omega_\epsilon \cap B_{R_0}$ it is weakly lower semi-continuous:
There exists a unique $\phi_\epsilon \in \Omega_\epsilon \cap B_{R_0}$
such that
\[
 I^\epsilon_W(\phi_\epsilon) = \inf_{\phi \in \Omega_\epsilon \cap B_{R_0}}
   I^\epsilon_W(\phi).
\]
The $\phi_-$-term in the definition of $I_W^\epsilon$ together with the
strict convexity of the functional $I_W^\epsilon$ ensures that
$\phi_\epsilon \geq 0$. Indeed, we see that
$I^\epsilon_W(\phi) \geq I^\epsilon_W(|\phi|)$. It follows from elliptic
regularity that $\phi_\epsilon \in W^{2, \frac{p}{2}}$ and from the
Harnack inequality that $\phi_\epsilon > 0$. In particular
$\phi_\epsilon \in \Omega_0 \cap B_{R_0}$.

To estimate the norm of $\phi_\epsilon$, we evaluate $I^\epsilon_W$ on
constant functions $\phi \equiv \lambda$:
\begin{align*}
 I^\epsilon_W(\lambda)
  & = \frac{\lambda^2}{2} \int_M \cR_\psi d\mu^g - \frac{\lambda^N}{N} \int_M \cB_{\tau, \psi} d\mu^g
   + \frac{(\lambda+\epsilon)^{-N}}{N} \int_M \left(\left|\sigma + \bL W\right|^2 + \pi^2\right) d\mu^g\\
  & = \frac{a}{2} \lambda^2 - \frac{b}{N} \lambda^N + \frac{c}{N} (\lambda+\epsilon)^{-N},
\end{align*}
where
\[
 a = \int_M \cR_\psi d\mu^g,\qquad
 b = \int_M \cB_{\tau, \psi} d\mu^g, \qquad
 c = \int_M \left(\left|\sigma + \bL W\right|^2 + \pi^2\right) d\mu^g.
\]

Some simple analysis shows that the stable minimum of $I_0(\lambda)$ is
attained at some value $\lambda \sim \left(\frac{c}{a}\right)^{\frac{1}{N+2}}$.
We thus have
\begin{align*}
 \Ibar(\phi_\epsilon)
  & \leq I^\epsilon_W(\phi_\epsilon)\\
  & \leq I^\epsilon_W \left(\left(\frac{c}{a}\right)^{\frac{1}{N+2}}-\epsilon\right)\\
  & \leq \left(\frac{1}{2} + \frac{1}{N}\right) \left(\frac{c^2}{a^N}\right)^{\frac{1}{N+2}}
    - \frac{b}{N} \left(\frac{c}{a}\right)^{\frac{N}{N+2}} - a \epsilon \left(\frac{c}{a}\right)^{\frac{1}{N+2}}
    + \frac{a}{2} \epsilon^2.
\end{align*}

Choosing $\epsilon \leq \left(\frac{c}{a}\right)^{\frac{1}{N+2}}$
and using Inequality \eqref{eqEstimateIbar}, we get
\[
 \frac{1}{4}\left\|\phi_\epsilon\right\|^2_h \leq \Ibar(\phi_\epsilon)
   \leq \left(\frac{1}{2}+\frac{1}{N}\right) \left(\frac{c^2}{a^N}\right)^{\frac{1}{N+2}}.
\]

It is important to remark at this point that the estimate we got for
$\left\|\phi_\epsilon\right\|^2_h$ is actually independent of $\epsilon$.

Following \cite{Premoselli}, we construct a (positive) sub-solution to
the equation for the critical points of the functional
\eqref{eqDefIepsilon}:
\begin{equation}\label{eqLichnerowiczRegularized}
  \frac{4(n-1)}{n-2} \Delta \phi + \cR_\psi \phi
   = \cB_{\tau, \psi} \phi^{N-1}
    + \frac{\left|\sigma + \bL W\right|^2 + \pi^2}{(\phi+\epsilon)^{N+1}}
\end{equation}
Note that the set $\Omega_\epsilon$ has empty interior in $H^1$ so one
cannot speak about the Hessian of $I_W^\epsilon$ restricted to this set.
Critical points are here to be understood as points for which the
G\^ateaux derivative of the functional $I_W^\epsilon$ vanishes in the
direction of smooth functions. Nevertheless Equation
\eqref{eqLichnerowiczRegularized} is satisfied by the function
$\phi_\epsilon$ as long as $\left\|\phi_\epsilon\right\|_h < R_0$
because one then has that $\phi_\epsilon + t \xi \in \Omega_{\epsilon} \cap B_{R_0}$
for any smooth function $\xi$ as long as $|t|$ is small enough.

Since the construction of a subsolution will be useful later, we collect
it in a lemma:
\begin{lemma}\label{lmSubSolution}
 There exists a positive subsolution $\phi_{\mathrm{sub}} \in W^{2, \frac{p}{2}}$
 independent of $\epsilon$ to all \eqref{eqLichnerowiczRegularized}.
 Further, $\phi_{\mathrm{sub}}$ can be chosen as small as we want in $H^1$. If
 $\phi_{\mathrm{sub}}$ and $\phi_\epsilon$ satisfy
 \[
  \left\|\phi_\epsilon\right\|_{L^N}, \left\|\phi_{\mathrm{sub}}\right\|_{L^N}
   < \left(\frac{s}{(N-1) \left\|\cB_{\tau, \psi}\right\|_{L^\infty}}\right)^{\frac{N}{N-2}},
 \]
  we have $\phi_\epsilon \geq \phi_{\mathrm{sub}}$.
\end{lemma}

\begin{proof}
 Defining $\cB_- \definedas \min \{\cB_{\tau, \psi}, 0\}$, and given some
 $\alpha$ to be chosen later, we solve the following equation for $u$:
 \begin{equation}\label{eqSubsolution}
  \frac{4(n-1)}{n-2} \Delta u + \cR_\psi u = \left|\sigma + \bL W\right|^2 + \pi^2 + \alpha \cB_-.
 \end{equation}
 Note that when $\alpha = 0$, this equation was already studied in
 Step \ref{stLichnerowicz0}, Section \ref{secGicquaudNgo}.
 The corresponding solution $u$ was continuous and positive, hence,
 choosing $\alpha > 0$ small enough, we still get a positive solution
 to \eqref{eqSubsolution}.

 We now set $\phi_{\mathrm{sub}} \definedas \theta u$ for some $\theta > 0$. As in
 the proof of Theorem \ref{thmGicquaudNgo}, it can be checked that,
 provided $\theta$ is small enough (depending only on $\max(u)$),
 $\phi_{\mathrm{sub}}$ is a subsolution to \eqref{eqLichnerowiczRegularized}, namely:
 \[
  \frac{4(n-1)}{n-2} \Delta \phi_{\mathrm{sub}} + \cR_\psi \phi_{\mathrm{sub}}
   \leq \cB_{\tau, \psi} \phi_{\mathrm{sub}}^{N-1}
    + \frac{\left|\sigma + \bL W\right|^2 + \pi^2}{(\phi_{\mathrm{sub}}+\epsilon)^{N+1}}.
 \]
 Indeed, the condition for $\psi_{\mathrm{sub}}$ to be a subsolution reads
 \[
  \left(\theta - \theta^{-N-1} u^{-N-1}\right) A_W + \alpha \theta \cB_- - \cB \theta^{N-1} u^{N-1} \leq 0,
 \]
 which follows from
 \[
  \theta \left(1 - \theta^{-N-2} u^{-N-1}\right) A_W + \cB_- \left(\alpha \theta - \theta^{N-1} u^{N-1}\right) \leq 0.
 \]
 This last condition is fulfilled by choosing $\theta > 0$ such that
 \[
  \left\lbrace
  \begin{array}{rcl}
   1 & \leq & \theta^{-N-2} u^{-N-1},\\
   \alpha \theta & \geq & \theta^{N-1} u^{N-1}
  \end{array}
  \right.
  \Leftrightarrow
  \left\lbrace
  \begin{array}{l}
   \theta \leq \left(\max u\right)^{\frac{N+1}{N+2}},\\
   \theta \leq \alpha^{\frac{1}{N-2}} \left(\max u\right)^{\frac{1-N}{N-2}}.
  \end{array}
  \right.
 \]

 We define $(\phi_\epsilon - \phi_{\mathrm{sub}})_- \definedas \min\{0, \phi_\epsilon - \phi_{\mathrm{sub}}\}$.
 Subtracting Equation \eqref{eqLichnerowiczRegularized} for $\phi_\epsilon$
 with the previous inequality satisfied by $\phi_{\mathrm{sub}}$, multiplying by
 $(\phi_\epsilon - \phi_{\mathrm{sub}})_-$ and integrating over $M$, we get:
 \begin{align*}
  & \int_M \left(\frac{4(n-1)}{n-2} \left|d(\phi_\epsilon-\phi_{\mathrm{sub}})_-\right|^2 + \cR_\psi (\phi_\epsilon-\phi_{\mathrm{sub}})_-^2\right) d\mu^g\\
  & \qquad \qquad \leq \int_M \cB_{\tau, \psi} \left(\phi_\epsilon^{N-1} - \phi_{\mathrm{sub}}^{N-1}\right)(\phi_\epsilon-\phi_{\mathrm{sub}})_- d\mu^g\\
  & \qquad \qquad \qquad + \int_M \left(\left|\sigma + \bL W\right|^2 + \pi^2\right)\left(\frac{1}{(\phi_\epsilon+\epsilon)^{N+1}} - \frac{1}{(\phi_{\mathrm{sub}}+\epsilon)^{N+1}}\right)(\phi_\epsilon-\phi_{\mathrm{sub}})_- d\mu^g,\\
  & \qquad \qquad \leq (N-1) \int_M \cB_{\tau, \psi} \left(\int_0^1 (t \phi_\epsilon + (1-t) \phi_{\mathrm{sub}})^{N-2} dt\right)(\phi_\epsilon-\phi_{\mathrm{sub}})_-^2 d\mu^g\\
  & \qquad \qquad \qquad - (N+1) \int_M \left(\left|\sigma + \bL W\right|^2 + \pi^2\right)\left(\int_0^1 (t \phi_\epsilon + (1-t) \phi_{\mathrm{sub}} + \epsilon)^{-N-2} dt\right)(\phi_\epsilon-\phi_{\mathrm{sub}})_-^2 d\mu^g,\\
  & \qquad \qquad \leq (N-1) \left\|\cB_{\tau, \psi}\right\|_{L^\infty} \left(\int_M \left(\int_0^1 (t \phi_\epsilon + (1-t) \phi_{\mathrm{sub}})^{N-2} dt\right)^{\frac{N}{N-2}} d\mu^g\right)^{\frac{N-2}{N}}\\
  & \qquad \qquad \qquad \qquad \times \left(\int_M (\phi_\epsilon-\phi_{\mathrm{sub}})_-^N d\mu^g\right)^{\frac{2}{N}},\\
  & s \left(\int_M (\phi_\epsilon-\phi_{\mathrm{sub}})_-^N d\mu^g\right)^{\frac{2}{N}}\\
  & \qquad \qquad \leq (N-1) \left\|\cB_{\tau, \psi}\right\|_{L^\infty} \left(\max \{\|\phi_\epsilon\|_{L^N}, \|\phi_{\mathrm{sub}}\|_{L^N}\}\right)^{\frac{N-2}{N}} \left(\int_M (\phi_\epsilon-\phi_{\mathrm{sub}})_-^N d\mu^g\right)^{\frac{2}{N}}.
 \end{align*}
 We conclude that if $(N-1) \left\|\cB_{\tau, \psi}\right\|_{L^\infty} \left(\max \{\|\phi_\epsilon\|_{L^N}, \|\phi_{\mathrm{sub}}\|_{L^N}\}\right)^{\frac{N-2}{N}} < s$,
 $(\phi_\epsilon - \phi_{\mathrm{sub}})_- \equiv 0$ which is equivalent
 to saying that $\phi_\epsilon \geq \phi_{\mathrm{sub}}$.
\end{proof}

We now let $\epsilon$ go to zero. From the fact that
$\Omega_0 \cap B_{R_0}$ is weakly compact, there exists
$\phi_0 \in \Omega_0 \cap B_{R_0}$ which is the weak limit of some
sequence $(\phi_{\epsilon_i})_{i \geq 0}$, where $\epsilon_i \to 0$.
We can also assume that $\phi_{\epsilon_i} \to \phi_0$ a.e..

Since all $\phi_\epsilon$ are greater than or equal to
$\phi_{\mathrm{sub}}$, we have $\phi_0 \geq \phi_{\mathrm{sub}}$ and
\[
 \frac{\left|\sigma + \bL W\right|^2 + \pi^2}{(\phi_{\epsilon_i} + \epsilon_i)^{N+1}}
 \to \frac{\left|\sigma + \bL W\right|^2 + \pi^2}{\phi_0^{N+1}}
 \text{ in } L^q 
\]
for any $q < \frac{p}{2}$ since
\[
 \frac{1}{(\phi_{\epsilon_i} + \epsilon_i)^{N+1}}
\]
is uniformly bounded in $L^\infty$ and
\[
 \frac{1}{(\phi_{\epsilon_i} + \epsilon_i)^{N+1}}
 \to \frac{1}{\phi_0^{N+1}} \text{ a.e.}.
\]

As a consequence $\phi_0$ satisfies the Lichnerowicz equation
\eqref{eqLichnerowicz} in a weak sense. Elliptic regularity shows that
$\phi_0 \in W^{2, \frac{p}{2}}$ and
$I_W^\epsilon(\phi_\epsilon) \to I(\phi_0)$. Since
$I_W^\epsilon \leq I$ on $B_{R_0} \cap \Omega_0$, we have
\[
 I_W^\epsilon(\phi_\epsilon) = \min_{B_{R_0} \cap \Omega_0} I_W^\epsilon \leq \inf_{B_{R_0} \cap \Omega_0} I.
\]
This means that $I(\phi_0) = \inf_{B_{R_0} \cap \Omega_0} I$:
$\phi_0$ is a minimizer for $I$. From the fact that
$B_{R_0} \cap \Omega_0$ is convex and $I$ is strictly convex on
$B_{R_0} \cap \Omega_0$, we deduce that $\phi_0$ is the unique positive
solution to the Lichnerowicz equation on $B_{R_0}$.

\subsection{The coupled system}\label{secCoupled}

We now study the coupled system. As in \cite{DahlGicquaudHumbert}, the
operator
\[
 \phi \mapsto \frac{3n-2}{n-1} \Delta \phi + \cR_\psi \phi
\]
naturally appears. We make the temporary assumption that this operator is
coercive and let $s'$ be some positive constant so that
\[
 \left\|u\right\|_k^2 \definedas \int_M \left(\frac{3n-2}{n-1} \left|du\right|^2 + \cR_\psi u^2\right) d\mu^g
  \geq s'\left\|u\right\|_{L^N}^{2}.
\]
We shall even assume that $\cR_\psi$ is positive in the proof. This
assumption can be removed by performing a conformal change of the metric
$g$, see \cite[Proposition 1]{CBIP072}, and working with the conformal
thin sandwich method which is explicitly conformally covariant and differs
from the conformal method by the introduction of a lapse function. We refer
the reader to \cite{MaxwellConformalMethod} for an extensive discussion
of this fact.

We are going to use a fixed point argument. Starting from $\phi_0 \in L^{Np}$,
we solve the vector equation \eqref{eqVector} with $\phi \equiv \phi_0$ and get
$W \in W^{2, q}$, where $\frac{1}{q} = \frac{1}{p} + \frac{1}{n}$ which we plug
in the Lichnerowicz equation. Assuming that $\bL W$ is small enough in $L^2$,
Theorem \ref{thmLich} yields a unique $\phi > 0$ in $B_{R_0} \subset H^1$, which
by elliptic regularity belongs to $W^{2, \frac{p}{2}} \subset L^\infty \subset L^{Np}$.
We call this mapping (wherever it is defined) $\Phi$.

We first prove the following lemma:

 \begin{lemma}\label{lmL2NBound}
  There exists a $\mu_0 > 0$ and a constant $R > 0$ such that, provided
 \[
  \int_M \left(|\sigma|^2 + \pi^2\right) d\mu^g < \mu_0,
 \]
 the mapping $\Phi$ is well defined on the set
 \[
  C := \left\{\phi \in L^{Np}, \int_M \phi^{2N} d\mu^g \leq R\right\}
 \]
 and $C$ is stable for the mapping $\Phi$.
\end{lemma}

In the course of the proof, we will use the following fact: There exists
a constant $\gamma > 0$ such that for any $W \in H^1$, we have
\[
 \int_M \left|\bL W\right|^2 d\mu^g \geq \gamma \left(\int_M \left|W\right|^N d\mu^g\right)^{2/N}.
\]

\begin{proof}
 We contract the vector equation with $W$ and integrate over $M$. We
 obtain:

 \begin{align*}
  \frac{1}{2} \int_M \left|\bL W\right|^2 d\mu^g
   & = - \frac{n-1}{n} \int_M \phi^N \left\<d\tau, W\right\> + \int_M \pi \left\<d\psi, W\right\> d\mu^g\\
   & \leq \frac{n-1}{2n} \left(\alpha \int_M \phi_0^{2N} d\mu^g + \frac{1}{\alpha} \int_M \left|d\tau\right|^2 \left|W\right|^2 d\mu^g\right)\\
   & \qquad + \frac{1}{2} \left(\beta \int_M \pi^2 d\mu^g + \frac{1}{\beta} \int_M \left|d\psi\right|^2\left|W\right|^2 d\mu^g\right)\\
   & \leq \frac{n-1}{2n} \left[\alpha \int_M \phi_0^{2N} d\mu^g + \frac{\left\|d\tau\right\|_{L^n}^{2/n}}{\alpha} \left(\int_M \left|W\right|^N d\mu^g\right)^{2/N}\right]\\
   & \qquad + \frac{1}{2} \left[\beta \int_M \pi^2 d\mu^g + \frac{\left\|d\psi\right\|_{L^n}^{2/n}}{\beta} \left(\int_M \left|W\right|^N d\mu^g\right)^{2/N}\right]\\
   & \leq \frac{n-1}{2n} \left[\alpha \int_M \phi_0^{2N} d\mu^g + \frac{\left\|d\tau\right\|_{L^n}^{2/n}}{\alpha\gamma} \int_M \left|\bL W\right|^2 d\mu^g\right]\\
   & \qquad + \frac{1}{2} \left[\beta \int_M \pi^2 d\mu^g + \frac{\left\|d\psi\right\|_{L^n}^{2/n}}{\beta\gamma} \int_M \left|\bL W\right|^2 d\mu^g\right],
 \end{align*}
 \begin{align*}
   \left(\frac{1}{2} - \frac{n-1}{2n} \frac{\left\|d\tau\right\|_{L^n}^{2/n}}{\alpha\gamma} - \frac{1}{2} \frac{\left\|d\psi\right\|_{L^n}^{2/n}}{\beta\gamma}\right)
   &  \int_M \left|\bL W\right|^2 d\mu^g\\
   & \qquad \leq \frac{n-1}{2n} \alpha \int_M \phi_0^{2N} d\mu^g + \frac{1}{2} \beta \int_M \pi^2 d\mu^g.
 \end{align*}
 Choosing $\alpha$ and $\beta$ large enough, we conclude that there exist constants $c_1$ and $c_2$
 such that
 \begin{equation}\label{eqEstimateVector}
   \int_M \left(\left|\sigma + \bL W\right|^2 + \pi^2 \right) d\mu^g
   \leq \int_M |\sigma|^2 d\mu^g + c_1 \int_M \phi_0^{2N} d\mu^g + (1+c_2) \int_M \pi^2 d\mu^g.
 \end{equation}
 This proves that, if $\phi_0 \in L^{2N}$ and $\pi, \sigma \in L^2$ are small
 enough, $\bL W$ is small in $L^2$ so Theorem \ref{thmLich} applies
 giving a solution $\phi$ to the Lichnerowicz equation.

 Next, we multiply the Lichnerowicz equation by $\phi^{N+1}$ and
 integrate by parts the Laplacian:
 \begin{align*}
   & \frac{3n-2}{n-1} \int_M \left|d\phi^{\frac{N}{2}+1}\right|^2 d\mu^g
      + \int_M \cR_\psi \phi^{N+2} d\mu^g\\
   & \qquad \qquad = \int_M \cB_{\tau, \psi} \phi^{2N} d\mu^g + \int_M \left(\left|\sigma + \bL W\right|^2 + \pi^2\right) d\mu^g\\
   & \qquad \qquad \leq \left\| \cB_{\tau, \psi}\right\|_{L^\infty} \int_M \phi^{2N} d\mu^g
      + \int_M \left(\left|\sigma\right|^2 + (1+c_2) \pi^2\right) d\mu^g + c_1 \int_M \phi_0^{2N} d\mu^g\\
   & \qquad \qquad \leq \left\| \cB_{\tau, \psi}\right\|_{L^\infty} \left(\int_M \phi^{N} d\mu^g\right)^{\frac{N-2}{N}} \left(\int_M \phi^{N\frac{N+2}{2}} d\mu^g\right)^{\frac{2}{N}}\\
   & \qquad \qquad \qquad + \int_M \left(\left|\sigma\right|^2 + (1+c_2) \pi^2\right) d\mu^g + c_1 \int_M \phi_0^{2N} d\mu^g.\\
 \end{align*}
 Hence, introducing the norm $\|\cdot\|_k$ (see \eqref{eqDefNormK}),
 \begin{equation}\label{eqBoundNormK}
 \begin{aligned}
  & \left(1-\frac{\left\| \cB_{\tau, \psi}\right\|_{L^\infty}}{(s')^{1/N}} \left(\int_M \phi^{N} d\mu^g\right)^{\frac{N-2}{N}}\right)\left\|\phi^{\frac{N}{2}+1}\right\|_k^2\\
  & \qquad \qquad \leq \int_M \left(\left|\sigma\right|^2 + (1+c_2) \pi^2\right) d\mu^g + c_1 \int_M \phi_0^{2N} d\mu^g.
 \end{aligned}
 \end{equation}
 From the Sobolev embedding together with the H\"older inequality,
 we get:
 \begin{align*}
  \int_M \phi^{2N} d\mu^g
   & \leq \vol(M, g)^{\frac{N-2}{N+2}} \left\|\phi^{1+\frac{N}{2}}\right\|_{L^N}^{\frac{4N}{N+2}}\\
   & \leq \vol(M, g)^{\frac{N-2}{N+2}} (s')^{-\frac{2N}{N+2}} \left\|\phi^{1+\frac{N}{2}}\right\|_k^{\frac{4N}{N+2}}.
 \end{align*}
 Thus, assuming that $\|\phi\|_{L^N}$ is small enough:
 \[
  \left(\int_M \phi^{N} d\mu^g\right)^{\frac{N-2}{N}} \leq \frac{1}{2} \frac{(s')^{1/N}}{\left\| \cB_{\tau, \psi}\right\|_{L^\infty}},
 \]
 we conclude that
 \begin{align*}
  & \frac{s'}{2} \vol(M, g)^{\frac{N-2}{2N}} \left(\int_M \phi^{2N} d\mu^g\right)^{\frac{N+2}{2N}}\\
  & \qquad \qquad \leq \int_M \left(\left|\sigma\right|^2 + (1+c_2) \pi^2\right) d\mu^g + c_1 \int_M \phi_0^{2N} d\mu^g.
 \end{align*}
 Denoting
 \[
  y = \int_M \phi^{2N} d\mu^g, \text{ resp. } y_0 = \int_M \phi_0^{2N} d\mu^g,
 \]
 we obtain an inequality of the following form for $y$:
 \begin{equation}\label{eqDefF0}
  y \leq \left(\frac{x}{\lambda} + \frac{c_1}{\lambda} y_0\right)^{\frac{2N}{N+2}},
 \end{equation}
 where
  \[
   x = \int_M \left(\left|\sigma\right|^2 + (1+c_2) \pi^2\right) d\mu^g \text{ and }
   \lambda = \frac{s'}{2} \vol(M, g)^{\frac{N-2}{2N}}.
  \]
 We denote by $f(y_0)$ the right-hand side of \eqref{eqDefF0}. Note that
 $f$ is an increasing function. We seek for some $R > 0$, $R \ll 1$ such
 that $f(R) \leq R$. This would have the consequence that the set
 \[
  C = \left\{\phi \in L^{Np}, \int_M \phi^{2N} d\mu^g \leq R\right\}
 \]
 is stable for the mapping $\Phi$. Indeed, we would then have that,
 given $\phi_0 \in C$,
 \[
  \int_M \phi^{2N} d\mu^g \leq f\left(\int_M \phi^{2N}_0 d\mu^g\right) \leq f(C) \leq C.
 \]
 By some simple Taylor expansion, we see that
 $C = 2 \left(\frac{x}{\lambda}\right)^{\frac{2N}{N+2}}$ does the job
 provided that $x$ is small enough.
\end{proof}

The remaining steps of the proof go as in \cite{MaxwellNonCMC}. There is
however a subtlety appearing here. Continuity of the solution $\phi$
of the Lichnerowicz equation \eqref{eqLichnerowicz} with respect to
$\bL W$ is usually obtain by the implicit function theorem. But the set
$\Omega_0 = \{\phi \in H^1, \phi \geq 0 \text{ a.e.}\}$ has empty interior.
Hence working on the set $C$ is not enough.

\begin{prop}\label{propLinftyBound}
 Assuming that
 \[
  \int_M \left(|\sigma|^2 + \pi^2\right) d\mu^g < \mu_0,
 \]
 where $\mu_0> 0$ is as defined in Lemma \ref{lmL2NBound}, 
 there exist sequences $(q_i)_{i \geq 0}$ and $(R_i)_{i \geq 0}$,
 $q_i \geq 2$, $q_i \to \infty$ and $R_i > 0$ such that, setting
 \[
  C_k \definedas C \cap \bigcap_{i = 0}^k \left\{\phi \in L^{Np}, \left\|\phi\right\|_{L^{Nq_i}} \leq R_i\text{ for all } i, 0 \leq i \leq k\right\},
 \]
 $\Phi$ maps $C_k$ into $C_{k+1} \subset C_k$.
\end{prop}

We use an induction argument which is quite similar in spirit to the one
used in \cite{DahlGicquaudHumbert, GicquaudSakovich}. Note that, however,
in these references, the Laplacian term is discarded because it vanishes
for large solutions. Here it will play an important role.

\begin{proof}[Proof of Proposition \ref{propLinftyBound}]
We define sequences $q_i\geq 2, R_i$ inductively so that
\[
 \sup \left\{\left\|\phi\right\|_{L^{Nq_i}}, \phi \in \Phi^i(C)\right\} \leq R_i.
\]
 From Lemma \ref{lmL2NBound}, we can choose $q_0 = 2$ and $R_0 = R^{1/2N}$.

 Given $\phi_0 \in C_i$, we set $\phi = \Phi(\phi_0)$. Note that
 $\phi_0 \in C_{i-1}$ (or $\phi_0 \in C$ if $i=0$), hence, by induction,
 $\phi \in C_i$ (when $i=0$, this is Lemma \ref{lmL2NBound}).

 The solution $W$ to the vector equation
 \[
  - \frac{1}{2} \bL^* \bL W = \frac{n-1}{n} \phi^N d\tau - \pi d\psi.
 \]
 belongs to $W^{2, r_i}$, where $r_i$ is given by
 \[
  \frac{1}{r_i} = \frac{1}{n} + \frac{1}{q_i} \geq \frac{1}{n}.
 \]
 By elliptic regularity, together with the Sobolev embedding,
 \begin{equation}\label{eqLsiBoundW}
 \begin{aligned}
  \left\|\bL W\right\|_{L^{q_i}}
   & \lesssim R_i \left\|d\tau\right\|_{L^n} + \left\|\pi d\psi\right\|_{L^{r_i}}\\
   & \lesssim R_i + \left\|\pi d\psi\right\|_{L^n}.
 \end{aligned}
 \end{equation}

 We multiply the Lichnerowicz equation for $\phi = \Phi(\phi_0)$
 by $\phi^{N+1+2 k_i}$ for some $k_i \geq 0$ to be chosen later and
 integrate over $M$ to get:
 \begin{equation}\label{eqIntegratedLich}
 \begin{aligned}
  & \frac{4(n-1)}{n-2} \frac{N+1+2k_i}{\left(\frac{N}{2}+1+k_i\right)^2} \int_M \left|d\left(\phi\right)^{\frac{N}{2}+1+k_i}\right|^2 d\mu^g
    + \int_M \cR_\psi \phi^{N+2+2k_i} d\mu^g\\
  & \qquad \qquad = \int_M \cB_{\tau, \psi} \phi^{2N+2k_i} d\mu^g
    + \int_M \left(\left|\sigma + \bL W\right|^2 + \pi^2\right)\phi^{2k_i}
 \end{aligned}
 \end{equation}
 Since we assumed that $\cR_\psi > 0$, there exists a constant $s_i > 0$
 so that
 \[
  \forall \xi \in H^1,
   \frac{4(n-1)}{n-2} \frac{N+1+2k_i}{\left(\frac{N}{2}+1+k_i\right)^2} \int_M \left|d\xi\right|^2 d\mu^g + \int_M \cR_\xi\phi^2 d\mu^g
    \geq s_i \left(\int_M \xi^N d\mu^g \right)^{\frac{2}{N}}.
 \]
 Applying this inequality to \eqref{eqIntegratedLich} with
 $\xi \equiv \phi^{\frac{N}{2}+1+k_i}$, we get:
 \begin{equation}\label{eqBoundPhi}
 \begin{aligned}
  & s_i \left(\int_M \phi^{N\left(\frac{N}{2}+1+k_i\right)} d\mu^g\right)^{2/N}\\
  & \qquad \qquad \leq \left\|\cB_{\tau, \psi}\right\|_{L^\infty}\int_M \phi^{2N+2k_i} d\mu^g
    + \int_M \left(\left|\sigma + \bL W\right|^2 + \pi^2\right) \phi^{2k_i}\\
  & \qquad \qquad \leq \left\|\cB_{\tau, \psi}\right\|_{L^\infty} \left(\int_M \phi^{2N} d\mu^g\right)^{1-x} \left(\int_M \phi^{2N+ \frac{2 k_i}{x}} d\mu^g\right)^x\\
  & \qquad \qquad \qquad + \left(\int_M \left(\left|\sigma + \bL W\right|^2 + \pi^2\right)^{\frac{q_i}{2}} d\mu^g\right)^{\frac{2}{q_i}}
     \left(\int_M \phi^{2 k_i \frac{q_i}{q_i-2}} d\mu^g\right)^{1-\frac{2}{q_i}},
 \end{aligned}
 \end{equation}
 where $x \in (0, 1)$ is some constant to be chosen later.
 From Equation \eqref{eqLsiBoundW}, we have that
 \begin{align*}
  \left(\int_M \left(\left|\sigma + \bL W\right|^2 + \pi^2\right)^{\frac{q_i}{2}} d\mu^g\right)^{\frac{2}{q_i}}
    & \lesssim \left\|\sigma\right\|_{L^p}^2 + \left\|\pi\right\|_{L^p}^2 + \left\|\bL W\right\|_{L^{q_i}}^2
 \end{align*}
 is bounded from above independently of $W$ by some constant $C_i$.
 We choose $k_i$ so that
  \[
  2 k_i \frac{q_i}{q_i-2} = N q_i,
 \]
 i.e.,
 \begin{equation}\label{eqCrap}
  k_i = \frac{N}{2} (q_i-2) \geq 0
 \end{equation}
 Note that since $\phi = \Phi(\phi_0) \in C$, we have that
 \[
  \left(\int_M \phi^{2 k_i \frac{q_i}{q_i-2}} d\mu^g\right)^{1-\frac{2}{q_i}} \leq R_i^{q_i-2}.
 \]
 We now come back to the choice of $x$. We let $x$ be such that
 \[
  2 N + \frac{2 k_i}{x} = N \left(\frac{N}{2}+1+k_i\right),
 \]
 that is to say
 \[
  x = \frac{2 k_i}{N k_i + N \left(\frac{N}{2}-1\right)} < \frac{2}{N}.
 \]
 We finally arrive at the following inequality:
 \[
  \begin{aligned}
  & s_i \left(\int_M \phi^{N\left(\frac{N}{2}+1+k_i\right)} d\mu^g\right)^{2/N}\\
  & \qquad \qquad \leq \left\|\cB_{\tau, \psi}\right\|_{L^\infty} R^{1-x} \left(\int_M \phi^{N\left(\frac{N}{2}+1+k_i\right)} d\mu^g\right)^x
    + C_i R_i^{q_i-2}.
  \end{aligned}
 \]
 Since $x < \frac{2}{N}$ we immediately deduce that, setting
 $q_{i+1} = \frac{N}{2}+1+k_i$,
 \[
  \left\|\phi\right\|_{L^{Nq_{i+1}}} \leq R_{i+1}
 \]
 for some $R_{i+1}$ independent of $\phi_0 \in C$. We have
 \[
  q_{i+1} = \frac{N}{2} + 1 + \frac{N}{2} (q_i-2) = \frac{N}{2} (q_i-1) + 1
 \]
 so $q_i = 1+ \left(\frac{N}{2}\right)^i$ goes to infinity with $i$.

 We point here that we were slightly sloppy. Namely for $i=0$, $k_0=0$ and
 $x=0$ which is not allowed in our calculation. Note however that
 multiplying the Lichnerowicz equation with $\phi^{2N}$ and integrating
 over $M$, we get, as in the proof of Lemma \ref{lmL2NBound}, that
 \[
  \begin{aligned}
   & \frac{3n-2}{n-1} \int_M \left|d\phi^{\frac{N}{2}+1}\right|^2 d\mu^g
      + \int_M \cR_\psi \phi^{N+2} d\mu^g\\
   & \qquad \qquad \leq \left\| \cB_{\tau, \psi}\right\|_{L^\infty} \int_M \phi^{2N} d\mu^g
      + \int_M \left(\left|\sigma + \bL W\right|^2 + \pi^2\right) d\mu^g\\
   & \qquad \qquad \leq R \left\| \cB_{\tau, \psi}\right\|_{L^\infty}
      + \int_M \left(\left|\sigma + \bL W\right|^2 + \pi^2\right) d\mu^g\\
  \end{aligned}
 \]
 so the argument still applies.
\end{proof}

We now choose $k$ so that $q_k \geq p$ and set
$\Cbar \definedas C_k$. We come back to the subsolution introduced
in Lemma \ref{lmSubSolution}. This lemma is taken from
\cite{MaxwellNonCMC}.

\begin{lemma}\label{lmSubSolution2}
 There exists $\eta > 0$ so that all $\phi = \Phi(\phi_0)$ with
 $\phi_0 \in \Cbar$ satisfy $\phi \geq \eta$.
\end{lemma}

\begin{proof}
 We study in more details the proof of Lemma \ref{lmSubSolution}. We can
 write $u = u_1 - \alpha u_2$, where $u_1$ and $u_2$ solve
 \[
 \left\lbrace
 \begin{aligned}
  \frac{4(n-1)}{n-2} \Delta u_1 + \cR_\psi u_1 & = \left|\sigma + \bL W\right|^2 + \pi^2,\\
  \frac{4(n-1)}{n-2} \Delta u_2 + \cR_\psi u_2 & = \cB_+.
 \end{aligned}
 \right.
 \]
 The Green function $G(x, y)$ of the modified conformal Laplacian
 \[
  \frac{4(n-1)}{n-2} \Delta + \cR_\psi
 \]
 is positive and continuous outside the diagonal of $M \times M$ where it
 blows up. Hence, there exists a constant $\epsilon > 0$ such that
 $G(x, y) \geq \epsilon$. This implies that
 \begin{align*}
  u_1(x) & = \int_M G(x, y) \left(\left|\sigma + \bL W\right|^2 + \pi^2\right)(y) d\mu^g(y)\\
         & \geq \epsilon \int_M \left(\left|\sigma + \bL W\right|^2 + \pi^2\right)(y) d\mu^g(y)\\
         & \geq \epsilon \int_M \left(\left|\sigma\right|^2 + \left|\bL W\right|^2 + \pi^2\right) d\mu^g\\
         & \geq \epsilon \int_M \left(\left|\sigma\right|^2 + \pi^2\right) d\mu^g.
 \end{align*}
 So $u_1$ is bounded from below independently of $W$ so $\alpha$ in
 the proof of Lemma \ref{lmSubSolution} can be chosen independently of $W$
 so that e.g.
 $u \geq \frac{\epsilon}{2} \int_M \left(\left|\sigma\right|^2 + \pi^2\right) d\mu^g$.
 Since we assumed that $\phi \in \Cbar$, we also have that
 \[\left|\sigma + \bL W\right|^2 + \pi^2\]
 is bounded from above by some constant depending on $R'$ in $L^{p/2}$
 so $u$ is bounded in $W^{2, \frac{p}{2}} \into L^\infty$ independently
 of the choice of $\phi \in \Cbar$.

 Hence, the constant $\theta$ so that $\phi_{\mathrm{sub}} = \theta u$ is a sub-solution
 to the Lichnerowicz equation can be chosen independently of $W$.

 Setting
 \[
  \eta = \frac{\epsilon \theta}{2} \int_M \left(\left|\sigma\right|^2 + \pi^2\right) d\mu^g,
 \]
 we have $\phi_{\mathrm{sub}} \geq \eta$ so $\phi \geq \eta$.
\end{proof}

\begin{lemma}
 Under the assumptions of the previous lemma, the mapping $\Phi: \Cbar \to \Cbar$
 is continuous and compact.
\end{lemma}

\begin{proof}
 We first prove continuity of the mapping $\Phi$. Assume given a sequence
 $(\phi_i)_i$, $\phi_i \in \Cbar$ such that $\phi_i \to \phi_\infty$ in
 $L^{Np}$.

 We denote with a prime their images under the mapping $\Phi$:
 $\phi'_i = \Phi(\phi_i)$, $\phi'_\infty = \Phi(\phi_\infty)$. And we
 also denote by $W_i$ (resp. $W_0$) the corresponding solutions to the
 vector equation:
 \[
  \left\lbrace
  \begin{aligned}
   -\frac{1}{2} \bL^* \bL W_i & = \frac{n-1}{n} \phi_i^N d\tau - \pi d\psi,\\
   -\frac{1}{2} \bL^* \bL W_\infty & = \frac{n-1}{n} \phi_\infty^N d\tau - \pi d\psi.
  \end{aligned}
  \right.
 \]
 We have $W_i \to W_\infty$ in $W^{2, q}$, $\frac{1}{q} = \frac{1}{p}+\frac{1}{n}$,
 so $\left|\bL W_i\right|^2 \to \left|\bL W_\infty\right|^2$ in $L^{\frac{p}{2}}$.
 Since the Hessian of $I_{W_\infty}$ is more coercive on $B_{R_0}$ than that of $\Ibar$,
 we have from Lemma \ref{lmIbar}:
 \[
  \frac{1}{4} \left\|\phi'_i-\phi'_\infty\right\|^2_h \leq I_{W_\infty}(\phi'_i) - I_{W_\infty}(\phi'_\infty)
 \]
 for some constant $\lambda > 0$. It follows from Lemma
 \ref{lmSubSolution2} that $\phi'_i \geq \eta$ for all $i$ (resp.
 $\phi'_\infty \geq \eta$). As a consequence,
 \begin{align*}
  \lambda \left\|\phi'_i-\phi'_\infty\right\|^2_h
   & \leq I_{W_\infty}(\phi'_i) - I_{W_\infty}(\phi'_\infty)\\
   & \leq \left(I_{W_\infty}(\phi'_i) - I_{W_i}(\phi'_i)\right) + \left(I_{W_i}(\phi'_i) - I_{W_\infty}(\phi'_\infty)\right)\\
   & \leq \frac{\eta^{-N}}{N} \left\|\left|\bL W_i\right|^2 - \left|\bL W_\infty\right|^2\right\|_{L^1} + \sup_{\phi \in B_{R_0} \cap \Omega_{-2\eta}} \left|I_{W_i}(\phi)-I_{W_\infty}(\phi)\right|\\
   & \leq 2 \frac{\eta^{-N}}{N} \left\|\left|\bL W_i\right|^2 - \left|\bL W_\infty\right|^2\right\|_{L^1},
 \end{align*}
 where to pass from the second line to the third, we used the fact that
 the map ``infimum'' is 1-Lipschitzian.
 Thus we get that $\phi'_i \to \phi'_\infty$ in $H^1$ and in particular
 $\phi'_i \to \phi'_\infty$ in $L^N$. Convergence in $L^{Np}$ follows
 from elliptic regularity. Indeed, looking at the Lichnerowicz equation
 for $\phi'_i$:
 \[
  \frac{4(n-1)}{n-2} \Delta \phi'_i + \cR_\psi \phi'_i
    = \cB_{\tau, \psi} (\phi'_i)^{N-1}
    + \frac{\left|\sigma + \bL W_i\right|^2 + \pi^2}{(\phi'_i)^{N+1}},
 \]
 we see that the righthand side is bounded in $L^{\frac{p}{2}}$
 independently of $i$, as a consequence of Lemma \ref{lmSubSolution2}. So
 the sequence $(\phi'_i)$ is bounded in $W^{2, \frac{p}{2}} \into L^\infty$.
 By interpolation, $(\phi'_i)$ is a Cauchy sequence in $L^{Np}$ whose
 limit in $L^N$ is $\phi'_\infty$. We conclude that $\phi'_i \to \phi'_\infty$
 in $L^{Np}$.

 Compactness of the mapping $\Phi$ is fairly simple since we noticed
 that the set $\left(\Phi(C)\right)^{\frac{N+2}{2}}$ is bounded in $H^1$
 (this is Estimate \eqref{eqBoundNormK}) so $\Phi(C)$ embeds compactly
 in $L^{2N}$ by the Rellich theorem. Then notice that pursuing one step
 further the proof of Proposition \ref{propLinftyBound}, the set
 $\Phi(\Cbar)$ is bounded in $L^{Nq_{K+1}}$. Compactness of
 $\Phi(\Cbar)$ for the $L^{Np}$-norm follows by interpolation.
\end{proof}

Theorem \ref{thmHolst} then follows by applying the Schauder fixed point
theorem. Namely, the convex hull of $\Phi(\Cbar) \subset \Cbar$ is
compact, convex and stable for the mapping $\Phi$. So $\Phi$ admits a
fixed point $\phi \in \Cbar$ which is in turn a solution to the conformal
constraint equations.

\providecommand{\bysame}{\leavevmode\hbox to3em{\hrulefill}\thinspace}
\providecommand{\href}[2]{#2}

\end{document}